\documentclass{amsart}

\usepackage{amssymb,amsmath,amsfonts,cite,graphicx,gastex}

\DeclareMathOperator{\var}{var}

\DeclareMathOperator{\ini}{ini}

\DeclareMathOperator{\occ}{occ}

\DeclareMathOperator{\con}{con}

\DeclareMathOperator{\FIC}{FIC}

\DeclareMathOperator{\Part}{Part}

\newtheorem{theorem}{Theorem}[section]

\newtheorem{proposition}[theorem]{Proposition}

\newtheorem{lemma}[theorem]{Lemma}

\newtheorem{corollary}[theorem]{Corollary}

\numberwithin{equation}{section}

\makeatletter

\renewcommand*\subjclass[2][2010]{\def\@subjclass{#2}\@ifundefined{subjclassname@#1}{\ClassWarning{\@classname}{Unknown edition (#1) of Mathematics Subject Classification; using '2010'.}}{\@xp\let\@xp\subjclassname\csname subjclassname@#1\endcsname}}

\renewcommand{\subjclassname}{\textup{2010} Mathematics Subject Classification}

\makeatother

\begin{document}

\title[Cancellable elements of the lattice of monoid varieties]{Cancellable elements of the lattice\\ of monoid varieties}

\thanks{The first author is supported by the Ministry of Science and Higher Education of the Russian Federation (project FEUZ-2020-0016)}

\author[S.\,V.~Gusev]{Sergey V.~Gusev}
\address{Institute of Natural Sciences and Mathematics, Ural Federal University, Lenina str.\@~51, 620000 Ekaterinburg, Russia}
\email{sergey.gusb@gmail.com}

\author[E.\,W.\,H.~Lee]{Edmond W.\,H.~Lee}
\address{Department of Mathematics, Nova Southeastern University, Fort Lauderdale, FL 33314, USA}
\email{edmond.lee@nova.edu}

\date{}

\begin{abstract}
The set of all cancellable elements of the lattice of semigroup varieties has recently been shown to be countably infinite.
But the description of all cancellable elements of the lattice $\mathbb{MON}$ of monoid varieties remains unknown.
This problem is addressed in the present article.
The first example of a monoid variety with modular but non-distributive subvariety lattice is first exhibited.
Then a necessary condition of the modularity of an element in $\mathbb{MON}$ is established.
These results play a crucial role in the complete description of all cancellable elements of the lattice $\mathbb{MON}$.
It turns out that there are precisely five such elements.
\end{abstract}

\keywords{Monoid, variety, lattice of varieties, cancellable element of a lattice, modular element of a lattice}

\subjclass{20M07, 08B15}

\maketitle

\section{Introduction and summary}
\label{Sec: introduction}

The present article is concerned with the lattice $\mathbb{MON}$ of all monoid varieties, where monoids are considered as semigroups with an identity element that is fixed by a \mbox{0-ary} operation.
For many years, results on the lattice $\mathbb{MON}$ were scarce.
But recently, interest in this lattice has grown significantly; in particular, the study of its special elements was initiated in the articles~\cite{Gusev-18AU,Gusev-21+}.
In the present work, we continue these investigations.

Special elements play an important role in general lattice theory; see~\cite[Section~III.2]{Gratzer-11}, for instance.
We recall definitions of those types of special elements that are relevant here.
An element $x$ of a lattice $L$ is
\begin{align*}
&\text{\textit{cancellable} if}&&\forall\,y,z\in L\colon\,\,\, x\vee y=x\vee z\ \&\ x\wedge y=x\wedge z\longrightarrow y=z;\\
&\text{\textit{modular} if}&&\forall\,y,z\in L\colon\,\,\, y\le z\longrightarrow (x\wedge z)\vee y = (x\vee y)\wedge z.
\end{align*}
It is easy to see that every cancellable element is modular.

Our main goal is to describe all cancellable elements of the lattice $\mathbb{MON}$.
To formulate our main result, we need some definitions and notation.
Let~$\mathfrak X^+$ [respectively, $\mathfrak X^\ast$] denote the free semigroup [respectively, monoid] over a countably infinite alphabet~$\mathfrak X$.
Elements of~$\mathfrak X$ are called \textit{letters} and elements of~$\mathfrak X^\ast$ are called \textit{words}.
Words unlike letters are written in bold.
An identity is written as $\mathbf u \approx \mathbf v$, where $\mathbf u,\mathbf v \in \mathfrak X^\ast$.

Let $\bf T$, $\bf SL$, and $\bf MON$ denote the variety of trivial monoids, the variety of semilattice monoids, and the variety of all monoids, respectively.
For any identity system $\Sigma$, let $\var\Sigma$ denote the variety of monoids given by $\Sigma$.
Put \[ \mathbf C_2 = \var\{x^2\approx x^3,\, xy\approx yx\}\ \text{ and }\ \mathbf D = \var\{x^2 \approx x^3,\, x^2y \approx xyx \approx yx^2\}. \]
Then the following is our main result.

\begin{theorem}
\label{T: cancel}
A monoid variety is a cancellable element of the lattice $\mathbb{MON}$ if and only if it coincides with one of the varieties $\mathbf T$, $\mathbf{SL}$, $\mathbf C_2$, $\mathbf D$ or $\mathbf{MON}$.
\end{theorem}

Many articles were devoted to special elements of different types in the lattice $\mathbb{SEM}$ of all semigroup varieties;
an overview of results published before 2015 can be found in the survey~\cite{Vernikov-15}.\footnote{An extended version of this survey, which is periodically updated as new results are found and/or new articles are published, is available at http://arxiv.org/abs/1309.0228v20.}
It is natural to compare Theorem~\ref{T: cancel} with the description of cancellable elements of the lattice $\mathbb{SEM}$ that was found in 2019 \cite{Shaprynskii-Skokov-Vernikov-19}.
Theorem~\ref{T: cancel} shows that there are only five cancellable elements in the lattice $\mathbb{MON}$.
In contrast, the set of all cancellable elements of the lattice $\mathbb{SEM}$ is countably infinite.

In general, the set of cancellable elements in a lattice need not form a sublattice.
For example, the elements $x$ and $y$ of the lattice in Fig.~\ref{join is not cancellable} are cancellable but their join $x\vee y$ is not.
However, the class of all cancellable elements of $\mathbb{SEM}$ forms a distributive sublattice of $\mathbb{SEM}$; see Corollary~3.14 in the extended version of the survey~\cite{Vernikov-15}.
Theorem~\ref{T: cancel} shows that the same is true for monoid varieties too; in fact, the five cancellable elements in $\mathbb{MON}$ constitute a chain.

\begin{figure}[htb]
\begin{center}
\unitlength=1mm
\linethickness{0.4pt}
\begin{picture}(60,31)

\put(30,0){\circle*{1.33}}

\put(10,10){\circle*{1.33}}
\put(50,10){\circle*{1.33}}

\put(30,10){\circle*{1.33}}
\put(30,20){\circle*{1.33}}

\put(20,15){\circle*{1.33}}
\put(40,15){\circle*{1.33}}

\put(10,20){\circle*{1.33}}
\put(50,20){\circle*{1.33}}
\put(30,30){\circle*{1.33}}
\gasset{AHnb=0,linewidth=0.4}
\drawline(30,0)(10,10)(10,20)(30,10)(50,20)(50,10)(30,0)
\drawline(10,20)(30,30)(50,20)
\drawline(30,0)(30,10)
\drawline(30,30)(30,20)(20,15)
\drawline(30,20)(40,15)
\put(19,12.5){\makebox(0,0){$x$}}
\put(41,12.5){\makebox(0,0){$y$}}
\end{picture}
\caption{}
\label{join is not cancellable}
\end{center}
\end{figure}

Now since the chain $\mathbf T \subset \mathbf{SL} \subset \mathbf C_2 \subset \mathbf D$ coincides with the lattice $\mathfrak L(\mathbf D)$ of subvarieties of $\mathbf D$ (see Fig.~\ref{L(R vee dual to R)}),
a monoid variety $\mathbf V$ is a cancellable element of the lattice $\mathbb{MON}$ if and only if either $\mathbf V \subseteq \mathbf D$ or $\mathbf V = \mathbf{MON}$.
It is routinely verified that the variety $\mathbf D$ can be given by the single identity $x^3yz \approx yxzx$.
Therefore it is easy to check the cancellability of proper elements of the lattice $\mathbb{MON}$; a monoid variety is \textit{proper} if it is different from $\mathbf{MON}$.

\begin{corollary}
Suppose that $M$ is any monoid that generates a proper subvariety $\mathbf V$ of $\mathbf{MON}$.
Then $\mathbf V$ is a cancellable element of the lattice $\mathbb{MON}$ if and only if $M$ satisfies the identity $x^3yz \approx yxzx$.
\end{corollary}

The article consists of five sections.
Section~\ref{Sec: preliminaries} contains definitions, notation, certain known results and their simple corollaries.
In Section~\ref{Sec: modular but non-distributive}, the first example of a monoid variety with modular but non-distributive subvariety lattice is given.
In Section~\ref{Sec: necessary condition}, a necessary condition of the modularity of an element in $\mathbb{MON}$ is established in Proposition~\ref{P: identities for mod}.
Results from Sections~\ref{Sec: modular but non-distributive} and~\ref{Sec: necessary condition} will then be used in Section~\ref{Sec: proof of main result} to prove Theorem~\ref{T: cancel}.

\section{Preliminaries}
\label{Sec: preliminaries}

Acquaintance with rudiments of universal algebra is assumed of the reader.
Refer to the monograph~\cite{Burris-Sankappanavar-81} for more information.

Recall that a variety is \textit{periodic} if it consists of periodic monoids.
Equivalently, a variety is periodic if and only if it satisfies the identity $x^n \approx x^{n+m}$ for some $n,m \ge 1$.
For any word $\mathbf w$ and any set $X$ of letters, the word obtained from $\mathbf w$ by deleting all the letters of $X$ is denoted by $\mathbf w_X$.
The \textit{content} of a word $\mathbf w$, denoted by $\con(\mathbf w)$, is the set of letters occurring in $\mathbf w$.
The partition lattice over a set $X$ is denoted by $\Part(X)$.
Let $\mathcal L_{\FIC(\mathfrak X^\ast)}$ denote the lattice of all fully invariant congruences on the monoid $\mathfrak X^\ast$, and for any variety $\mathbf V$ of monoids, let $\FIC(\mathbf V)$ denote the fully invariant congruence on $\mathfrak X^\ast$ corresponding to $\mathbf V$.
It is well known that the mapping $\FIC\colon\mathbb{MON}\longrightarrow\mathcal L_{\FIC}(\mathfrak X^\ast)$ is an anti-isomorphism of lattices; see \cite[Theorem~11.9 and Corollary~14.10]{Burris-Sankappanavar-81}, for instance.
For any $\mathbf u, \mathbf v \in \mathfrak X^+$, we put $\mathbf u\preceq \mathbf v$ if $\mathbf v = \mathbf a\xi(\mathbf u)\mathbf b$ for some words $\mathbf a, \mathbf b \in \mathfrak X^\ast$ and some endomorphism $\xi$ of $\mathfrak X^+$.
It is easily seen that the relation $\preceq$ on $\mathfrak X^+$ is a quasi-order.
For an arbitrary anti-chain $A\subseteq \mathfrak X^+$ under the relation $\preceq$, let $\mathrm L_A$ denote the set of all monoid varieties $\mathbf V$ for which $A$ is a union of $\FIC(\mathbf V)$-classes.
Define the map $\varphi_A\colon\mathrm L_A\longrightarrow \Part(A)$ by the rule $\varphi_A(\mathbf V)=\FIC(\mathbf V)|_A$ for any $\mathbf V\in\mathrm L_A$.

\begin{lemma}[{\!\cite[Lemma~3]{Gusev-18IzVUZ}}]
\label{L: anti-hom}
Let $A$ be any anti-chain under the quasi-order $\preceq$.
Suppose that for any words $\mathbf u,\mathbf v\in A$ and any nonempty set $X\subseteq \con(\mathbf u)$, the equalities $\con(\mathbf u)=\con(\mathbf v)$ and $\mathbf u_X = \mathbf v_X$ hold.
Then
\begin{itemize}
\item[\textup{(i)}] the set $\mathrm L_A$ is a sublattice of the lattice $\mathbb{MON}$;
\item[\textup{(ii)}] the map $\varphi_A$ is a surjective anti-homomorpism of the lattice $\mathrm L_A$ onto the lattice $\Part(A)$;
\item[\textup{(iii)}] for any partition $\beta\in\Part(A)$, there exists a non-periodic monoid variety $\mathbf V\in \mathrm L_A$ such that $\varphi_A(\mathbf V)=\beta$.
\end{itemize}
\end{lemma}

Recall that a band is \textit{left regular} if it is a semilattice of left zero bands.
It is well known that the class of left regular band monoids coincides with the variety \[ \mathbf{LRB}=\var\{xy \approx xyx\}. \]
The \textit{initial part} of a word $\mathbf w$, denoted by $\ini(\mathbf w)$, is the word obtained from $\mathbf w$ by retaining the first occurrence of each letter.
The following assertion is well known and easily verified.

\begin{lemma}
\label{word problem LRB}
An identity $\mathbf u \approx \mathbf v$ holds in $\mathbf{LRB}$ if and only if $\ini(\mathbf u) = \ini(\mathbf v)$.
\end{lemma}

For any $n\ge2$, the variety \[ \mathbf C_n = \var\{x^n\approx x^{n+1},\, xy\approx yx\} \] is generated by the monoid $\langle a, 1 \mid a^n = 0 \rangle$ \cite[Corollary~6.1.5]{Almeida-94}.
Note that the variety $\mathbf C_2$ has already been introduced in Section~\ref{Sec: introduction}.
A word $\mathbf w$ is an \textit{isoterm} for a variety $\mathbf V$ if the $\FIC(\mathbf V)$-class of $\mathbf w$ is singleton.
The following result is easily deduced from \cite[Lemma~3.3]{Jackson-05}.

\begin{lemma}
\label{L: x^n is isoterm}
Let $n \ge 1$. For any monoid variety $\mathbf V$, the following are equivalent:
\begin{itemize}
\item[a)] $x^n$ is not an isoterm for $\mathbf V$;
\item[b)] $\mathbf V$ satisfies the identity $x^n \approx x^{n+m}$ for some $m \ge 1$;
\item[c)] $\mathbf C_{n+1}\nsubseteq\mathbf V$.
\end{itemize}
\end{lemma}

A monoid is \textit{completely regular} if it is a union of its maximal subgroups.
A variety is \textit{completely regular} if it consists of completely regular monoids.
It is well known that a monoid variety is completely regular if and only if it satisfies the identity $x\approx x^{n+1}$ for some $n \ge 1$.

\begin{lemma}[{\!\cite[Lemma~2.14]{Gusev-Vernikov-18}}]
\label{L: non-cr and non-commut}
If a monoid variety $\mathbf V$ is non-completely regular and noncommutative, then $\mathbf D\subseteq \mathbf V$.
\end{lemma}

\begin{lemma}
\label{L: does not contain E}
Let $\mathbf V$ be any monoid variety such that $\mathbf C_2\subseteq\mathbf V$.
Suppose that $\mathbf V$ does not contain the variety \[ \mathbf E=\var\{x^2 \approx x^3,\, x^2y \approx xyx,\, x^2y^2 \approx y^2x^2\}. \]
Then $\mathbf V$ satisfies the identity $x^pyx^q \approx yx^r$ for some $p,q \ge 1$ and $r \ge 2$.
\end{lemma}

\begin{proof}
If $\mathbf D\subseteq\mathbf V$, then the result follows from \cite[Lemma~4.1 and Proposition~4.2]{Gusev-Vernikov-18}.
Therefore suppose that $\mathbf D\nsubseteq\mathbf V$, so that by Lemma~\ref{L: non-cr and non-commut},
the variety $\mathbf V$ is either completely regular or commutative.
But $\mathbf V$ cannot be completely regular because $\mathbf C_2 \subseteq \mathbf V$.
Hence $\mathbf V$ is commutative and satisfies the identity $xyx\approx yx^2$.
\end{proof}

\section{Monoid variety with modular but non-distributive subvariety lattice}
\label{Sec: modular but non-distributive}

There are many examples of monoid varieties with non-distributive subvariety lattice; see~\cite{Gusev-18AU,Gusev-21+,Lee-12b}, for instance.
However, all these varieties have non-modular subvariety lattice as well.
In this section, we present the first example of a monoid variety whose subvariety lattice is modular but non-distributive.
To this end, the following varieties are required: the variety $\mathbf D_2$ generated by the monoid
\[
\langle a,b,1 \mid a^2=b^2=bab=0\rangle = \{ a,b,ab,ba,aba,1,0\},
\]
the variety $\mathbf R$ generated by the monoid \[ \langle a,b,1\mid a^3=b^2=ba=0\rangle=\{a,b,a^2,ab,a^2b,1,0\} \]
and the variety $\mathbf R^\delta$ dual to $\mathbf R$. It is proved in~\cite[Lemmas~2.2.8 and~2.2.9]{Jackson-99} that
\begin{align*}
\mathbf D_2=\var\{&x^3 \approx x^2,\, x^3yzt\approx yxzxtx,\\
&xyzxty \approx yxzxty,\,xzxyty \approx xzyxty,\, xtyzxy \approx xtyzyx\}, \\
\mathbf R\vee\mathbf R^\delta=\var\{&x^4\approx x^3,\,x^3yzt\approx yxzxtx,\\
&xyzxty \approx yxzxty,\,xzxyty \approx xzyxty,\, xtyzxy \approx xtyzyx\}.
\end{align*}
It is easily seen that $\mathbf D_2 = (\mathbf R\vee\mathbf R^\delta)\wedge\var\{x^3\approx x^2\}$.

\begin{proposition}
\label{P: mod not distributive}
The lattice $\mathfrak L(\mathbf R\vee\mathbf R^\delta)$ of subvarieties of $\mathbf R\vee\mathbf R^\delta$ is given in Fig.~\textup{\ref{L(R vee dual to R)}}.
In particular, this lattice is modular but not distributive.
\end{proposition}

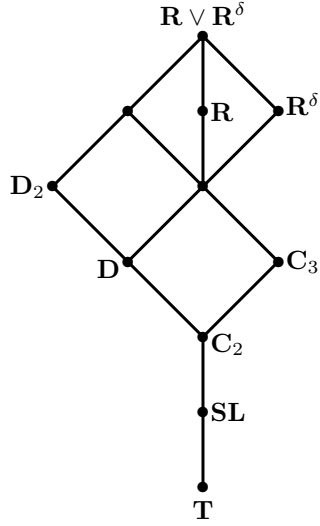
\begin{figure}[htb]
\unitlength=1mm
\linethickness{0.4pt}
\begin{center}
\begin{picture}(44,69)
\put(6,45){\circle*{1.33}}
\put(16,35){\circle*{1.33}}
\put(16,55){\circle*{1.33}}
\put(26,5){\circle*{1.33}}
\put(26,15){\circle*{1.33}}
\put(26,25){\circle*{1.33}}
\put(26,45){\circle*{1.33}}
\put(26,55){\circle*{1.33}}
\put(26,65){\circle*{1.33}}
\put(36,35){\circle*{1.33}}
\put(36,55){\circle*{1.33}}
\gasset{AHnb=0,linewidth=0.4}
\drawline(16,35)(26,45)(26,65)
\drawline(16,55)(26,45)
\drawline(26,5)(26,25)(6,45)(26,65)(36,55)(26,45)(36,35)(26,25)
\put(27,24){\makebox(0,0)[lc]{$\mathbf C_2$}}
\put(37,35){\makebox(0,0)[lc]{$\mathbf C_3$}}
\put(15,34){\makebox(0,0)[rc]{$\mathbf D$}}
\put(5,45){\makebox(0,0)[rc]{$\mathbf D_2$}}
\put(27,55){\makebox(0,0)[lc]{$\mathbf R$}}
\put(26,68){\makebox(0,0)[cc]{$\mathbf R\vee\mathbf R^\delta$}}
\put(37,56){\makebox(0,0)[lc]{$\mathbf R^\delta$}}
\put(27,15){\makebox(0,0)[lc]{$\mathbf{SL}$}}
\put(26,2){\makebox(0,0)[cc]{$\mathbf T$}}
\end{picture}
\end{center}
\caption{The subvariety lattice $\mathfrak L(\mathbf R\vee\mathbf R^\delta)$}
\label{L(R vee dual to R)}
\end{figure}

\begin{proof}
It is easily shown that $\mathbf C_3\subseteq\mathbf R\vee\mathbf R^\delta$.
According to Lemma~\ref{L: x^n is isoterm}, any subvariety $\mathbf V$ of $\mathbf R\vee\mathbf R^\delta$ such that $\mathbf C_3\nsubseteq\mathbf V$ satisfies the identity $x^3\approx x^2$, whence $\mathbf V\subseteq\mathbf D_2$.
Therefore, the lattice $\mathfrak L(\mathbf R\vee\mathbf R^\delta)$ is the disjoint union of the lattice $\mathfrak L(\mathbf D_2)$ and the interval $[\mathbf C_3,\mathbf R\vee\mathbf R^\delta]$.
It is proved in~\cite[Lemmas~4.4 and~4.5]{Jackson-05} that the lattice $\mathfrak L(\mathbf D_2)$ coincides with the 5-element chain in Fig.~\ref{L(R vee dual to R)}.
Thus it remains to describe the interval $[\mathbf C_3,\mathbf R\vee\mathbf R^\delta]$.
It follows from~\cite[Proposition~4.1]{Lee-12a} that every noncommutative variety in this interval is defined within $\mathbf R\vee\mathbf R^\delta$ by some of the identities $xyx\approx x^2y$, $xyx\approx yx^2$ or $x^2y\approx yx^2$.
It is then routinely shown that the interval $[\mathbf C_3,\mathbf R\vee\mathbf R^\delta]$ is as described in Fig.~\ref{L(R vee dual to R)}, where
\begin{align*}
\mathbf R&=(\mathbf R\vee\mathbf R^\delta)\wedge\var\{xyx\approx yx^2\},\\
\mathbf D_2\vee\mathbf C_3&=(\mathbf R\vee\mathbf R^\delta)\wedge\var\{x^2y\approx yx^2\},
\end{align*}
and $\mathbf D\vee\mathbf C_3=(\mathbf D_2\vee\mathbf C_3)\wedge\mathbf R=(\mathbf D_2\vee\mathbf C_3)\wedge\mathbf R^\delta=\mathbf R\vee\mathbf R^\delta$.
The proof of this proposition is thus complete.
\end{proof}

\section{Necessary condition of the modularity of an element in $\mathbb{MON}$}
\label{Sec: necessary condition}

Given any word $\mathbf w$ and letter $x$, let $\occ_x(\mathbf w)$ denote the number of occurrences of $x$ in $\mathbf w$.
Let $\lambda$ denote the empty word.
Let $W = W_1 \cup W_2$, where
\begin{align*}
W_1 & =\{y^{r_1}xt^{r_2}z^{r_3}y^{r_4}t^{r_5} xz^{r_6}\mid r_1,r_2,r_3,r_4,r_5,r_6\ge2\},\\
W_2 & =\{y^{r_1}xt^{r_2}z^{r_3}xy^{r_4}t^{r_5} xz^{r_6}\mid r_1,r_2,r_3,r_4,r_5,r_6\ge2\}.
\end{align*}
Let us fix the following two words: \[ \mathbf p = y^2xt^2z^2y^2t^2xz^2\ \text{ and }\ \mathbf q = y^2xt^2z^2xy^2t^2xz^2. \]
Put $\mathbf K = \var\{\mathbf p \approx \mathbf q\}$.

\begin{lemma}
\label{L: FIC(K)-class}
The set $W$ is a $\FIC(\mathbf K)$-class.
\end{lemma}

\begin{proof}
Let $\mathbf u \approx \mathbf v$ be any identity of $\mathbf K$ with $\mathbf u\in W$.
We need to verify that $\mathbf v\in W$.
By assumption, there is a deduction of the identity $\mathbf u\approx\mathbf v$ from the identity $\mathbf p\approx\mathbf q$, that is, a sequence $\mathbf w_0, \mathbf w_1, \ldots, \mathbf w_m$ of words such that $\mathbf w_0=\mathbf u$, $\mathbf w_m=\mathbf v$ and, for each $i=0,1,\dots,m-1$, there are words $\mathbf a_i,\mathbf b_i\in \mathfrak X^\ast$ and an endomorphism $\xi_i$ of $\mathfrak X^\ast$ such that $\mathbf w_i=\mathbf a_i\xi_i(\mathbf s_i)\mathbf b_i$ and $\mathbf w_{i+1}=\mathbf a_i\xi_i(\mathbf t_i)\mathbf b_i$, where $\{\mathbf s_i,\mathbf t_i\}=\{\mathbf p,\mathbf q\}$.
By trivial induction on $m$, it suffices to only consider the case when $\mathbf u=\mathbf a\xi(\mathbf s)\mathbf b$ and $\mathbf v=\mathbf a\xi(\mathbf t)\mathbf b$ for some words $\mathbf a,\mathbf b\in \mathfrak X^\ast$, an endomorphism $\xi$ of $\mathfrak X^\ast$ and words $\mathbf s$ and $\mathbf t$ such that $\{\mathbf s,\mathbf t\}=\{\mathbf p,\mathbf q\}$.

Since any subword of $\mathbf u$ of the form $ab$, where $a$ an $b$ are distinct letters, occurs only once in $\mathbf u$ and all letters occurring in $\mathbf s$ are multiple, the following holds:
\begin{enumerate}
\item[(I)] For any $a\in\con(\mathbf s)$, either $\xi(a) = \lambda$ or $\xi(a)$ is a power of some letter.
\end{enumerate}
Further, since $\occ_x(\mathbf u)\le 3$ and $\occ_y(\mathbf s)=\occ_z(\mathbf s)=\occ_t(\mathbf s)=4$, we have
\begin{enumerate}
\item[(II)] $x \notin \con(\xi(yzt))$.
\end{enumerate}

We note that if $\xi(\mathbf s)=\lambda$ or $\xi(\mathbf s)$ is a power of some letter, then the required statement is evident.
So, we may assume that
\begin{enumerate}
\item[(III)] $|\con(\xi(\mathbf s))|\ge2$.
\end{enumerate}
Let $\mathbf u=y^{\ell_1}xt^{\ell_2}z^{\ell_3}x^cy^{\ell_4}t^{\ell_5} xz^{\ell_6}$, where $c\in\{0,1\}$ and $\ell_1,\ell_2,\ell_3,\ell_4,\ell_5,\ell_6\ge2$, and let
\[
d =
\begin{cases}
0&\text{if }\mathbf s=\mathbf p,\\
1&\text{if }\mathbf s=\mathbf q.
\end{cases}
\]
If $\xi(x)=\lambda$, then $\xi(\mathbf s)=\xi(\mathbf t)$, whence $\mathbf v=\mathbf u\in W$.
So, it remains to consider the case when $\xi(x)\ne\lambda$.
Then (I) implies that $\xi(x)$ is a power of some letter.

Suppose that $\xi(x)$ is a power of $y$.
Then (III) implies that $\con(\xi(t^2z^2x^dy^2t^2))$ contains one of the letters $x$, $z$ and $t$.
This is only possible when $\xi(t^2z^2x^dy^2t^2)=y^pxt^{\ell_2}z^{\ell_3}x^cy^q$ for some $0\le p\le \ell_1$ and $0\le q\le \ell_4$.
But since $x \notin \{ y \} = \con(\xi(x))$ by assumption and $x \notin \con(\xi(yzt))$ by (II), the contradiction $x\notin\con(\xi(t^2z^2x^dy^2t^2))$ is deduced.
Therefore, $\xi(x)$ cannot be a power of $y$.
Similarly, $\xi(x)$ cannot be a power of $z$ as well.

Suppose now that $\xi(x)$ is a power of $t$.
Then (III) implies that $\con(\xi(t^2z^2x^dy^2t^2))$ contains one of the letters $x$, $y$ and $z$.
This is only possible when $\xi(t^2z^2x^dy^2t^2)=t^pz^{\ell_3}x^cy^{\ell_4}t^q$ for some $0\le p\le \ell_2$ and $0\le q\le \ell_5$.
Then by (I), either $\xi(t) = \lambda$ or $\xi(t)$ is a power of $t$.
This implies that $\xi(z^2x^dy^2)=z^{\ell_3}x^cy^{\ell_4}$.
Taking into account that $\xi(x)$ is a power of $t$, we apply (I) again and obtain that $\xi(z^2)=z^{\ell_3}$, $\xi(y^2)=y^{\ell_4}$ and $c=d=0$.
This is only possible when $\xi(xt^2z^2x^d)=y^rxt^{\ell_2}z^{\ell_3}x^c$ for some $0\le r\le \ell_1$.
But since $x \notin \{ t \} = \con(\xi(x))$ by assumption and $x \notin \con(\xi(yzt))$ by (II), the contradiction $x\notin\con(\xi(xt^2z^2x^d))$ is deduced.
Therefore, $\xi(x)$ cannot be a power of $t$.

Finally, suppose that $\xi(x)$ is a power of $x$.
Then since $x^2$ is not a subword of $\mathbf u$, we have $\xi(x)=x$.

Suppose that $c=0$.
Then $d=0$ because otherwise, $\occ_x(\mathbf u)<\occ_x(\xi(\mathbf s))$.
Then $\xi(t^2z^2y^2t^2)=t^{\ell_2}z^{\ell_3}y^{\ell_4}t^{\ell_5}$.
It follow from (I) that $\xi(z^2)=z^{\ell_3}$, $\xi(y^2)=y^{\ell_4}$ and $\xi(t^2)=t^{\ell_2}=t^{\ell_5}$.
Then $\xi(\mathbf s)=y^{\ell_4}xt^{\ell_2}z^{\ell_3}y^{\ell_4}t^{\ell_5}xz^{\ell_3}$, $\mathbf a = y^{\ell_1-\ell_4}$ and $\mathbf b = z^{\ell_6-\ell_3}$.
Therefore, $\xi(\mathbf t)=y^{\ell_4}xt^{\ell_2}z^{\ell_3}xy^{\ell_4}t^{\ell_5}xz^{\ell_3}$, whence $\mathbf v= y^{\ell_1}xt^{\ell_2}z^{\ell_3}xy^{\ell_4}t^{\ell_5}xz^{\ell_6}\in W$, and we are done.

Suppose now that $c=1$.
If $x\in\con(\mathbf b)$, then $d=0$ because otherwise, $\occ_x(\mathbf u)<\occ_x(\xi(\mathbf s)\mathbf b)$.
This is only possible when
\[ \mathbf a\xi(y^2)=y^{\ell_1},\ x\xi(t^2z^2y^2t^2)x=xt^{\ell_2}z^{\ell_3}x\ \text{ and }\ \xi(z^2)\mathbf b = y^{\ell_4}t^{\ell_5}xz^{\ell_6}. \]
The second equality implies that $\xi(t^2z^2y^2t^2)=t^{\ell_2}z^{\ell_3}$.
Clearly, $\xi(t^2)=\lambda$, whence $\xi(z^2y^2)=t^{\ell_2}z^{\ell_3}$.
In view of (I), we have $\xi(z^2)=t^{\ell_2}$ and $\xi(y^2)=z^{\ell_3}$.
But this contradicts the fact that $\xi(z^2)\mathbf b = z^{\ell_4}t^{\ell_5}xz^{\ell_6}$.
Therefore, $x\notin\con(\mathbf b)$.
Analogously, one can verify that $x\notin\con(\mathbf a)$.
It follows that $d=1$.
Then
\[ \mathbf a\xi(y^2)=y^{\ell_1},\ x\xi(t^2z^2)x\xi(y^2t^2)x=xt^{\ell_2}z^{\ell_3}xy^{\ell_4}t^{\ell_5}x\ \text{ and }\ \xi(z^2)\mathbf b = z^{\ell_6}. \]
It follows from (I) that $\xi(z^2)=z^{\ell_3}$, $\xi(y^2)=y^{\ell_4}$ and $\xi(t^2)=t^{\ell_2}=t^{\ell_5}$.
Then $\xi(\mathbf s)=y^{\ell_3}xt^{\ell_2}z^{\ell_3}xy^{\ell_4}t^{\ell_5}xz^{\ell_4}$, $\mathbf a = y^{\ell_1-\ell_3}$ and $\mathbf b = z^{\ell_6-\ell_4}$.
Therefore, $\xi(\mathbf t)=y^{\ell_3}xt^{\ell_2}z^{\ell_3}xy^{\ell_4}t^{\ell_5}xz^{\ell_4}$, whence $\mathbf v= y^{\ell_1}xt^{\ell_2}z^{\ell_3}y^{\ell_4}t^{\ell_5}xz^{\ell_6}\in W$, and we are done.
\end{proof}

For any $n\ge1$, put \[ \mathbf B_n=\var\{x^n\approx x^{n+1}\}. \]

\begin{lemma}
\label{L: V periodic}
Suppose that $\mathbf V$ is any proper monoid variety that is a modular element of the lattice $\mathbb{MON}$.
Then $\mathbf V$ is periodic.
\end{lemma}

\begin{proof}
Seeking a contradiction, suppose that $\mathbf V$ is not periodic, so that $\mathbf V$ contains the variety $\mathbf{COM}$ of all commutative monoids.
Since $\mathbf V$ is proper and non-periodic, it satisfies some nontrivial identity $\mathbf u \approx \mathbf v$ such that every letter from $\con(\mathbf u\mathbf v)$ occurs $n$ times on both sides for some $n \geq 1$, that is, $n = \occ_a(\mathbf u) = \occ_a(\mathbf v)$ for all $a \in \con(\mathbf u\mathbf v)$.
Then by \cite[Lemma~3.2]{Sapir-00}, there exist two distinct letters $x$ and $y$ such that the identity obtained from $\mathbf u \approx \mathbf v$ by retaining $x$ and $y$ is nontrivial.
Therefore we may assume that $\con(\mathbf u)=\con(\mathbf v)=\{x,y\}$ with $n = \occ_x(\mathbf u) = \occ_x(\mathbf v) = \occ_y(\mathbf u) = \occ_y(\mathbf v)$.

Suppose that $\mathbf{LRB}\subseteq\mathbf V$.
In view of Lemma~\ref{word problem LRB}, we may assume without loss of generality that $\ini(\mathbf u) = \ini(\mathbf v)=xy$.
Let $\mathbf u'$ and $\mathbf v'$ be words that obtain from $\mathbf u$ and $\mathbf v$, respectively, by making the substitution $(x,y) \mapsto (y,x)$.
Then $\ini(\mathbf u') = \ini(\mathbf v')=yx$.
Put \[ A=\{\mathbf w\in\{x,y\}^+\mid\occ_x(\mathbf w)=n+1,\,\occ_y(\mathbf w)=n\}. \]
Let $\mathbf w,\mathbf w' \in A$ and $\mathbf w \preceq \mathbf w'$.
This means that $\mathbf w' = \mathbf a\xi(\mathbf w)\mathbf b$ for some words $\mathbf a, \mathbf b \in \mathfrak X^\ast$ and some endomorphism $\xi$ of $\mathfrak X^+$.
Since the length of $\mathbf w$ equals to the length of $\mathbf w'$, we have $\mathbf a = \mathbf b = \lambda$.
Then $\mathbf w' = \xi(\mathbf w)$.
But this is only possible when $\xi(x)=x$ and $\xi(y)=y$ because $\occ_y(\mathbf w)<\occ_x(\mathbf w)$.
Hence $\mathbf w = \mathbf w'$.
So, $A$ is an anti-chain under the quasi-order $\preceq$.
Then $\mathrm L_A$ is a sublattice of $\mathbb{MON}$ by Lemma~\ref{L: anti-hom}(i) and $\mathbf V\in \mathrm L_A$.

Clearly, $\mathbf ux,\mathbf u'x,\mathbf vx,\mathbf v'x\in A$.
Evidently, $\ini(\mathbf ux) = \ini(\mathbf vx)=xy$ and $\ini(\mathbf u'x) = \ini(\mathbf v'x)=yx$.
In view of Lemma~\ref{word problem LRB}, the words $\mathbf ux$ and $\mathbf u'x$ lie in distinct $\FIC(\mathbf V)$-classes.
Then, since $\mathbf V$ satisfies the nontrivial identities $\mathbf ux \approx \mathbf vx$ and $\mathbf u'x \approx \mathbf v'x$, the equivalence $\gamma=\varphi(\mathbf V)$ contains at least two non-singleton classes.
It is verified in~\cite[Proposition~2.2]{Jezek-81} that a partition $\rho\in\Part(X)$ is a modular element in $\Part(X)$ if and only if $\rho$ has at most one non-singleton class.
This result implies that $\gamma$ is not a modular element of the lattice $\Part(A)$.
Then there are $\alpha,\beta \in \Part(A)$ such that $\alpha\subset\beta$ and
\begin{equation}
\label{gamma is not modular}
(\gamma\wedge\beta)\vee\alpha\subset(\gamma\vee\alpha)\wedge\beta.
\end{equation}
According to Lemma~\ref{L: anti-hom}, we can find a non-periodic variety $\mathbf X \in \mathrm L_A$ such that $\varphi(\mathbf X)=\alpha$.
Put \[ \mathbf Y = \mathbf X \wedge \var\{\mathbf w \approx \mathbf w' \mid (\mathbf w,\mathbf w')\in \beta\}. \]
Clearly, $\mathbf Y\in \mathrm L_A$ and $\varphi(\mathbf Y)=\beta$.
Then \[ (\mathbf V\wedge\mathbf X)\vee\mathbf Y\subset(\mathbf V\vee\mathbf Y)\wedge\mathbf X \] because otherwise, the inclusion~\eqref{gamma is not modular} does not hold.
We see that $\mathbf V$ is not a modular element of the lattice $\mathbb{MON}$, which is a contradiction.

Suppose now that $\mathbf{LRB}\nsubseteq\mathbf V$.
Then Lemma~\ref{word problem LRB} allows us to assume that $\mathbf u$ starts with the letter $x$ but $\mathbf v$ starts with the letter $y$.
Let \[ \mathbf Z=\var\{x^{n+1}\approx x^{n+2},\,x^n\mathbf v \approx x^{n+1}\mathbf v\}. \]
We note that $\mathbf V\wedge\mathbf B_{n+1}\subseteq\mathbf Z$.
Indeed, $\mathbf V\wedge\mathbf B_{n+1}$ satisfies the identities \[ x^n\mathbf v\approx x^n\mathbf u\approx x^{n+1}\mathbf u\approx x^{n+1}\mathbf v \] and so the identity $x^n\mathbf v \approx x^{n+1}\mathbf v$.
Clearly, the word $x^n\mathbf v$ is an isoterm for both $\mathbf V\vee\mathbf Z$ and $\mathbf B_{n+1}$.
It follows that $x^n\mathbf v$ is an isoterm for $(\mathbf V\vee\mathbf Z)\wedge\mathbf B_{n+1}$ as well.
However, $x^n\mathbf v$ is not an isoterm for $\mathbf Z$ because $\mathbf Z$ satisfies $x^n\mathbf v \approx x^{n+1}\mathbf v$.
Therefore, \[ (\mathbf V\wedge\mathbf B_{n+1})\vee\mathbf Z=\mathbf Z\subset(\mathbf V\vee\mathbf Z)\wedge\mathbf B_{n+1}. \]
This means that $\mathbf V$ is not a modular element of the lattice $\mathbb{MON}$, which again is a contradiction.
\end{proof}

The following is the main result of this section.

\begin{proposition}
\label{P: identities for mod}
Suppose that $\mathbf V$ is any proper monoid variety that is a modular element of the lattice $\mathbb{MON}$.
Then $\mathbf V$ satisfies the identities
\begin{align}
\label{xx=xxx}
x^2&\approx x^3,\\
\label{xxy=yxx}
x^2y&\approx y x^2.
\end{align}
\end{proposition}

\begin{proof}
By Lemma~\ref{L: V periodic}, the variety $\mathbf V$ is periodic and so it satisfies the identity $x^n \approx x^{n+m}$ for some $n,m \ge 1$; we may assume $n$ and $m$ to be the least possible.

First, suppose that $n=1$, so that $\mathbf V$ is completely regular.
If $\mathbf X$ is a noncommutative completely regular variety, then it is verified in~\cite[Lemma~3.1]{Gusev-18AU} that
\[ (\mathbf X\wedge\mathbf D)\vee\mathbf C_2\subset(\mathbf X\vee\mathbf C_2)\wedge\mathbf D, \]
whence $\mathbf X$ is not a modular element of the lattice $\mathbb{MON}$.
If $\mathbf X$ is a commutative variety containing a nontrivial group, then it is proved in~\cite[Lemma~3.2]{Gusev-18AU} that
\[ (\mathbf X\wedge\mathbf B_2)\vee\mathbf Q\subset(\mathbf X\vee\mathbf Q)\wedge\mathbf B_2, \]
where $\mathbf Q=\var\{yxyzxy \approx yxzxyxz\}$, whence $\mathbf X$ is again not a modular element of the lattice $\mathbb{MON}$.
In view of these two facts, the variety $\mathbf V$ is commutative and does not contain any nontrivial group.
Since $\mathbf V$ is also completely regular, it is idempotent and so is contained in $\mathbf{SL}$.
Obviously, $\mathbf{SL}$ satisfies~\eqref{xx=xxx} and~\eqref{xxy=yxx}.

So, it remains to consider the case when $n>1$.
Then $\mathbf C_n\subseteq\mathbf V$ and $\mathbf C_{n+1}\nsubseteq \mathbf V$ by Lemma~\ref{L: x^n is isoterm}.
It follows from \cite[Lemma~2]{Gusev-21+} that $\mathbf E\nsubseteq\mathbf V$.
Then by Lemma~\ref{L: does not contain E}, $\mathbf V$ satisfies the identity $x^{p_1}yx^{q_1} \approx yx^{r_1}$ for some $p_1,q_1 \ge 1$ and $r_1 \ge 2$.
The dual arguments imply that $\mathbf V$ also satisfies the identity $x^{p_2}yx^{q_2} \approx x^{r_2}y$ for some $p_2,q_2 \ge 1$ and $r_2 \ge 2$.
Since one can substitute $x^n$ for $x$ in these identities and $\mathbf V$ satisfies $x^n \approx x^{n+m}$, we may assume without loss of generality that \[ p_1, p_2,q_1,q_2,r_1,r_2\in\{n,n+1,\dots,n+m-1\}. \]
Evidently, there exist $\ell_1$ and $\ell_2$ such that the identities
\[ x^{p_1}yx^{q_1+\ell_1} \approx yx^{r_1+\ell_1}\ \text{ and }\ x^{p_2+\ell_2}yx^{q_2} \approx x^{r_2+\ell_2}y \]
are equivalent modulo $x^n \approx x^{n+m}$ to the identities
\[ x^{p_1}yx^{q_2} \approx yx^{r_1+\ell_1}\ \text{ and }\ x^{p_1}yx^{q_2} \approx x^{r_2+\ell_2}y, \]
respectively.
Therefore $\mathbf V$ satisfies $x^{r_2+\ell_2}y \approx yx^{r_1+\ell_1}$, whence it satisfies
\begin{equation}
\label{x^ky=yx^k}
x^ky \approx yx^k
\end{equation}
for some $k \ge n$.
It follows that the meet $\mathbf V\wedge\mathbf B_2$ satisfies the identities \eqref{xx=xxx} and \eqref{x^ky=yx^k}; it also satisfies the identity $\mathbf p \approx \mathbf q$ because
\begin{align*}
\mathbf p = y^2xt^2z^2y^2t^2xz^2 & \stackrel{\eqref{xx=xxx}}\approx y^kxt^kz^ky^kt^kxz^k \\ & \stackrel{\eqref{x^ky=yx^k}}\approx x^2y^{2k}z^{2k}t^{2k} \\
& \stackrel{\eqref{xx=xxx}}\approx x^3y^{2k}z^{2k}t^{2k} \\ & \stackrel{\eqref{x^ky=yx^k}}\approx y^kxt^kz^kxy^kt^kxz^k \\ & \stackrel{\eqref{xx=xxx}}\approx y^2xt^2z^2xy^2t^2xz^2=\mathbf q.
\end{align*}
Therefore $\mathbf V\wedge\mathbf B_2\subseteq\mathbf K$, so that $(\mathbf V\wedge\mathbf B_2)\vee\mathbf K=\mathbf K$.

Suppose that $n>2$ or $m>1$.
Recall from the beginning of the section that \[ W_1=\{y^{r_1}xt^{r_2}z^{r_3}y^{r_4}t^{r_5} xz^{r_6}\mid r_1,r_2,r_3,r_4,r_5,r_6\ge2\}. \]
Let $\mathbf a \approx \mathbf b$ be any identity of $\mathbf V\vee\mathbf K$ with $\mathbf a\in W_1$.
If $n>2$, then $\mathbf b\in W_1$ by Lemmas~\ref{L: x^n is isoterm} and~\ref{L: FIC(K)-class}.
Clearly, $\mathbf V$ contains the variety $\mathbf A_m$ of all Abelian groups of exponent $m$.
It is well known and easily verified that an identity $\mathbf w \approx \mathbf w'$ holds in $\mathbf A_m$ if and only if $\occ_a(\mathbf w)\equiv\occ_a(\mathbf w')$ (mod $m$) for all $a \in \mathfrak X$.
This fact and Lemma~\ref{L: FIC(K)-class} imply that if $m>1$, then $\mathbf b\in W_1$.
We see that if $n>2$ or $m>1$, then $\mathbf b\in W_1$ in either case.
Evidently, if $\mathbf B_2$ satisfies an identity $\mathbf c \approx \mathbf d$ with $\mathbf c\in W_1$, then $\mathbf d\in W_1$.
This implies that if an identity of the form $\mathbf p \approx \mathbf w$ holds in $(\mathbf V\vee\mathbf K)\wedge \mathbf B_2$, then $\mathbf w\in W_1$.
In particular, $(\mathbf V\vee\mathbf K)\wedge \mathbf B_2$ violates $\mathbf p \approx \mathbf q$.
Therefore, \[ (\mathbf V\wedge\mathbf B_2)\vee\mathbf K=\mathbf K\subset(\mathbf V\vee\mathbf K)\wedge\mathbf B_2. \]
This means that $\mathbf V$ is not a modular element in $\mathbb{MON}$.
It follows that $n=2$ and $m=1$.
Then $\mathbf V$ satisfies~\eqref{xx=xxx}.
Besides that, since~\eqref{x^ky=yx^k} holds in the variety $\mathbf V$, this variety satisfies~\eqref{xxy=yxx}.

Proposition~\ref{P: identities for mod} is thus proved.
\end{proof}

\section{Proof of Theorem~\ref{T: cancel}}
\label{Sec: proof of main result}

\textit{Necessity}.
Let $\mathbf V$ be any proper monoid variety that is a cancellable element of the lattice $\mathbb{MON}$.
Since any cancellable element is modular, Proposition~\ref{P: identities for mod} implies that $\mathbf V$ satisfies the identities~\eqref{xx=xxx} and~\eqref{xxy=yxx}.
If $\mathbf V$ does not coincide with any of the varieties $\mathbf T$, $\mathbf{SL}$, $\mathbf C_2$ and $\mathbf D$, then $\mathbf V$ contains the variety $\mathbf D_2$ by~\cite[Lemma~3.3(i)]{Gusev-Vernikov-21+}.
Proposition~\ref{P: mod not distributive} and the fact that $\mathbf C_3\nsubseteq\mathbf V$ imply that $\mathbf V\vee\mathbf R = \mathbf V \vee \mathbf R^\delta$ and $\mathbf V\wedge\mathbf R = \mathbf V \wedge \mathbf R^\delta = \mathbf D$, contradicting the assumption that $\mathbf V$ is a cancellable element of $\mathbb{MON}$.
Hence $\mathbf V$ coincides with one of the varieties $\mathbf T$, $\mathbf{SL}$, $\mathbf C_2$ and $\mathbf D$.

\smallskip

\textit{Sufficiency}. Obviously, $\mathbf T$ and $\mathbf{MON}$ are cancellable elements of $\mathbb{MON}$.
An element $x$ of a lattice $L$ is \textit{costandard} if \[ \forall\,y,z\in L\colon\quad (x\wedge z)\vee y = (x\vee y)\wedge (z\vee y). \]
It is easily seen that any costandard element is cancellable.
It is shown in \cite[Theorem~1.2]{Gusev-18AU} that the varieties $\mathbf{SL}$ and $\mathbf C_2$ are costandard elements of the lattice $\mathbb{MON}$.
Therefore, these varieties are cancellable elements of this lattice.

So, it remains to establish that $\mathbf D$ is a cancellable element in $\mathbb{MON}$.
Let $\mathbf X$ and $\mathbf Y$ be monoid varieties such that $\mathbf D\vee\mathbf X=\mathbf D\vee\mathbf Y$ and $\mathbf D\wedge\mathbf X=\mathbf D\wedge\mathbf Y$.
If $\mathbf D\subseteq\mathbf X$, then $\mathbf D=\mathbf D\wedge\mathbf X=\mathbf D\wedge\mathbf Y$, so that $\mathbf D\subseteq\mathbf Y$, whence $\mathbf X=\mathbf D\vee\mathbf X=\mathbf D\vee\mathbf Y=\mathbf Y$ and we are done.
Therefore by symmetry, we may assume that $\mathbf D \nsubseteq \mathbf X$ and $\mathbf D \nsubseteq \mathbf Y$.

Now the subvariety lattice $\mathfrak L(\mathbf D)$ is the chain $\mathbf T\subset\mathbf{SL}\subset\mathbf C_2\subset\mathbf D$; see Fig.~\ref{L(R vee dual to R)}.
It follows that $\mathbf D\wedge \mathbf X=\mathbf D\wedge \mathbf Y\in\{\mathbf T,\mathbf{SL},\mathbf C_2\}$.
If $\mathbf D\wedge \mathbf X=\mathbf D\wedge \mathbf Y=\mathbf T$, then $\mathbf X$ and $\mathbf Y$ are varieties of groups by~\cite[Lemma~2.1]{Gusev-Vernikov-18}.
Then $\mathbf X\vee\mathbf Y$ is a variety of groups too and so $\mathbf{SL}\nsubseteq\mathbf X\vee\mathbf Y$, whence
\begin{equation}
\label{D wedge (X vee Y)}
\mathbf D \wedge (\mathbf X\vee\mathbf Y)=\mathbf D\wedge \mathbf X=\mathbf D\wedge \mathbf Y.
\end{equation}
If $\mathbf D\wedge \mathbf X=\mathbf D\wedge \mathbf Y=\mathbf{SL}$, then $\mathbf X$ and $\mathbf Y$ are completely regular varieties by~\cite[Corollary~2.6]{Gusev-Vernikov-18}.
Then $\mathbf X\vee\mathbf Y$ is completely regular and so $\mathbf C_2\nsubseteq\mathbf X\vee\mathbf Y$, whence the equality~\eqref{D wedge (X vee Y)} is true.
Finally, if $\mathbf D\wedge \mathbf X=\mathbf D\wedge \mathbf Y=\mathbf C_2$, then $\mathbf X$ and $\mathbf Y$ are commutative by Lemma~\ref{L: non-cr and non-commut}.
Then $\mathbf X\vee\mathbf Y$ is commutative and so $\mathbf D\nsubseteq\mathbf X\vee\mathbf Y$, whence the equality~\eqref{D wedge (X vee Y)} is true again.
We see that the equality~\eqref{D wedge (X vee Y)} holds in any case.

Clearly,
\begin{equation}
\label{D vee (X vee Y)}
\mathbf D\vee(\mathbf X\vee\mathbf Y)=\mathbf D\vee\mathbf X=\mathbf D\vee\mathbf Y.
\end{equation}
Then
\begin{align*}
\mathbf X & = (\mathbf D\wedge \mathbf X)\vee \mathbf X&&\text{because } \mathbf D\wedge \mathbf X\subset \mathbf X\\
& = (\mathbf D\wedge (\mathbf X \vee \mathbf Y))\vee \mathbf X&&\text{by~\eqref{D wedge (X vee Y)}}\\
& = (\mathbf D\vee\mathbf X)\wedge (\mathbf X \vee \mathbf Y)&&\text{by~\cite[Proposition~7]{Gusev-21+}}\\
& = (\mathbf D\vee(\mathbf X\vee\mathbf Y))\wedge (\mathbf X \vee \mathbf Y)&&\text{by~\eqref{D vee (X vee Y)}}\\
& = (\mathbf X \vee \mathbf Y)&&\text{because } \mathbf X \vee \mathbf Y\subset \mathbf D\vee(\mathbf X\vee\mathbf Y).
\end{align*}
We see that $\mathbf X=\mathbf X\vee\mathbf Y$.
By symmetry, $\mathbf Y=\mathbf X\vee\mathbf Y$, whence $\mathbf X=\mathbf Y$.
Therefore, $\mathbf D$ is a cancellable element in $\mathbb{MON}$. \qed

\end{document}